\documentclass[12pt]{article}
\usepackage{amsfonts}
\usepackage{amsmath}
\usepackage{amssymb}
\usepackage{amsthm}
\usepackage{xfrac}
\usepackage{bm}
\usepackage{graphicx}

\usepackage{authblk}

\newtheorem{remark}{Remark}
\newtheorem{theorem}{Theorem}
\newtheorem{corollary}{Corollary}

\newcommand{\rd}{\ensuremath{\mathrm{d}}}
\newcommand{\bydef}{:=}

\begin{document}

\title{Kinetic formulation of compartmental epidemic models}

\author[1]{Carolina Strecht-Fernandes\footnote{cst.fernandes@campus.fct.unl.pt, ORCID-ID: 0009-0008-1467-7633}}
\author[1]{Fabio A. C. C. Chalub\footnote{Corresponding author: facc@fct.unl.pt, ORCID-ID: 0000-0002-8081-9221}}

\affil[1]{Center for Mathematics and Applications (NOVA Math) and
Department of Mathematics, NOVA School of Science and Technology, Universidade NOVA de Lisboa, Quinta da Torre, 2829-516, Caparica, Portugal.}
\date{\today}

\maketitle

\begin{abstract}
We introduce a kinetic model that couples the movement of a population of individuals with the dynamics of a pathogen in the same population. We consider that transmission occurs when a susceptible and an infectious individual are sufficiently close for a sufficiently long time. We show that the model is formally compatible with the well-known SIRS model in mathematical epidemiology. Namely, after identifying an appropriate dimensionless variable, and considering the limit when that variable is small, we introduce a partial differential equation model of advection-drift-diffusion type (mesoscopic model), which for spatially homogeneous solutions reduces to the SIRS model. We prove the existence and uniqueness of solutions in appropriate spaces for particular instances of the model. We finish with some examples and discuss possible applications and generalisation of this modelling approach, linking kinetic models, evolutionary game theory, and mathematical epidemiology.
\end{abstract}

\smallskip


\section{Introduction}

Kinetic models were introduced in the 19th century in the study of gas dynamics~\cite{Cercignani_69,Cercignani_88}; in the late 20th century, many applications in biological modeling were found, from which we highlight applications in chemotaxis~\cite{Patlak_53,OthmerDunbarAlt_88,OthmerStevens_97,Chalubetal_04,Chalubetal_06}, immunology~\cite{DellaMarca_22,Oliveira_24}, cancer dynamics~\cite{Bellomo_04,Belloquid_04}, and, more recently, epidemiology~\cite{Boscherietal_21A,Bertagliaetal_21B}.

Mathematical models of epidemiology have a long story that dates back to seminal works by Daniel Bernoulli in the 18th Century, followed closely by Jean d'Alembert. In these works, they studied the possible impact of variolation, a primitive technique akin to vaccination (but with some important differences); cf.~\cite{Bacaer} for a detailed historical account with modern mathematical notation. After a long gap, in the early 20th Century, what is now known as compartmental models gained maturity. We refer to~\cite{Ross_1916,Kermack_1927} for the original references, and~\cite{Thieme,Britton,Murray_vol1} for a modern approach.

In the present work, we consider the so-called SIRS model, in which the population is divided into three classes: the \textbf{S}usceptible, individuals that can be contaminated by a certain pathogen, the \textbf{I}nfectious, individuals that can transmit the pathogen, and the \textbf{R}ecovered, individuals that, after being infected, benefit from a temporary or permanent immunity from the disease; when the immunity is permanent, it is customary to refer simply to the ``SIR model''. 

Transitions between the three classes above are given by $S+I\stackrel{\beta}{\to} 2I$, $I\stackrel{\gamma}{\to}R$, and $R\stackrel{\alpha}{\to}S$, where $\beta>0$, $\gamma>0$, and $\alpha\ge 0$, are the transmission rate, recovery rate, and loss of immunity rate, respectively.

One of the most well-known realisations of the above model is through the use of ordinary differential equations (ODE), in which case it reads
\begin{equation}\label{eq:SIRS}
\begin{cases}
    S'&=\alpha R-\beta\frac{SI}{S+I+R}\ ,\\
    I'&=\beta\frac{SI}{S+I+R}-\gamma I\ ,\\
    R'&=\gamma I-\alpha R\ .
\end{cases}
\end{equation}
This model is referred to as SIRS in the current work, including, by abuse of notation, the case $\alpha=0$. Several of its mathematical properties can be found in the references above.

We start by introducing a model inspired by the kinetic theory of gases, in which individuals move in a certain space; the movement is composed of a regression toward the mean, perturbed by a drift in a certain preferred direction that possibly depends on the full distribution of individuals (e.g., to avoid clusters of infectious individuals). Furthermore, depending on the encounters in state space (i.e, the space including not only position but also displacement velocity), individuals change their state from \textbf{S} to \textbf{I}, to \textbf{R}, and back to \textbf{S}. However, after introducing space and velocity in the SIRS model, transitions will depend not only on the total amount of individuals in each class, as in Eqs.~\eqref{eq:SIRS}, but on their relative positions and movement. More precisely, transitions from \textbf{S} to \textbf{I}, i.e., disease transmission, will be more likely if individuals are close to each other (at the same point of space), and if their contact lasts longer (i.e., with null relative motion). See Fig~\ref{fig:model} for an illustration of a kinetic model in two-dimensions.

\begin{figure}
    \centering
    \includegraphics[width=0.9\linewidth]{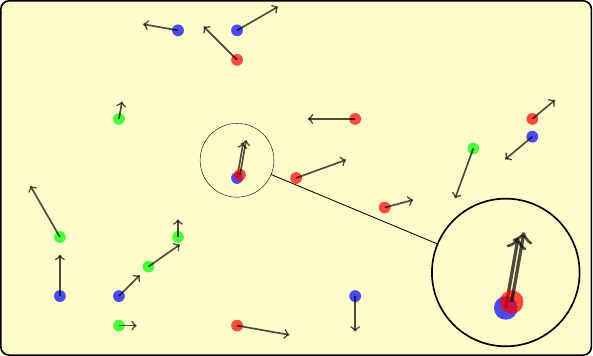}
    \caption{Illustration of the kinetic models for three classes of particles moving in a two-dimensional space. Each circle represents an individual, and the arrow indicates his or her velocity vector. Movements are composed of regression toward the mean velocity perturbed by a density-dependent drift. Blue individuals are of \textbf{S} type, while red and green individuals are of \textbf{I} and \textbf{R} types, respectively. In the highlight, we see an encounter in phase space (i.e., same position in space and same velocity) of an \textbf{S} and an \textbf{I} individual, possibly resulting in two \textbf{I} individuals, according to the prescribed transition rate. Note that when two individuals have the same position and velocity, there is more opportunity (i.e., time) for the transmission to happen than if only the same position were required.}
    \label{fig:model}
\end{figure}

We benefit from two recently introduced ideas in the field of evolutionary games. In~\cite{Ambrosio_etal_2021}, a kinetic model is formulated with a transition kernel akin to the replicator dynamics in evolutionary dynamics, i.e., individuals change classes according to the value obtained, or payoff, in their interaction with the remaining individuals in the population. In this work, however, the velocity is not directly related to physical space movement, but rather an internal parameter. See~\cite{HofbauerSigmund} for a detailed description of evolutionary game dynamics, in particular of the replicator equation.

In~\cite{HansenChalub_JTB2024}, the SIRS evolutionary dynamics was written in the framework of evolutionary game theory, showing that the SIRS model can be reinterpreted as a mix of two different games, with two different group sizes. 

Putting these two ideas together, we have a kinetic model with a transition kernel inspired by epidemic models. From the kinetic model, and following classical procedures, we compute the so-called \emph{macroscopic limit}. In the specialised literature, it is known as \emph{hydrodynamical limit}. Namely, it consists of the derivation of fluid dynamic equations from the detailed molecular dynamics modelled by the Boltzmann equation. On one hand, from a detailed microscopic description of the dynamics, relevant scales for problem variables are identified, including, e.g., time, space, and densities. On the other hand, macroscopic, observational variables are compared with typical microscopic variables, which are assumed to be comparatively small. This allows the derivation of simpler, but still precise, models with a focus on observable variables. See~\cite{Cercignani_69,Cercignani_88,SaintRaymond_2009,Golse_Levermore,Golse} and references therein for further information.

In this work, contrary to the usual in the literature, we will call the partial differential equations model a \emph{mesoscopic model}, and the \emph{macroscopic limit} will be reserved for the ordinary differential equation model, namely the SIR model, as it is our final goal.

There are several differences between the model introduced here and the one presented in~\cite{Boscherietal_21A} that are worth a brief comment. The most important one is that we include class transitions directly inspired by the SIRS model, with the aim of having as a macroscopic limit, when the problem is spatially homogeneous, the system~\eqref{eq:SIRS}. A second important difference is to consider more general movements of the individuals and proceed to a careful scale separation of the dependent and independent relevant variables and the class transitions. Finally, we study the existence and uniqueness of solutions of the kinetic model, although in a limited manner.

We finish this introduction with an outline of the present work. In Sec.~\ref{sec:model} we introduce our basic model, i.e, the kinetic model for epidemiology; in Sec.~\ref{sec:limit} we formally show that its form macroscopic limit is, in the spatially homogeneous case, the SIRS model~\eqref{eq:SIRS}, validating the idea that the kinetic model is a generalisation of the SIRS model. Furthermore, we define a mesoscopic model, a partial differential equation (PDE) of drift-diffusion-advection type, that lies in between the kinetic and the SIRS model.  In Sec. ~\ref{sec:existence}, we proceed to show the existence and uniqueness of solutions in the appropriate spaces of particular instances of the kinetic and, by extension, of the PDE model derived before. We show some applications in Sec.~\ref{sec:examples} and discuss possible limitations and continuations of the present work in Sec.~\ref{sec:conclusions}.

\section{The Model}
\label{sec:model}

Let $ f_j : \mathbb{R}_{+} \times \Omega \times V $, $ j = S, I, R$, be the density of individuals of class $j$ at time $ t \geq 0$, position $ x \in \Omega \subset \mathbb{R}^n $ and velocity $v \in V \subset  \mathbb{R}^n$. We assume that the space domain $\Omega$ is a bounded open set with smooth boundary and $V$ is compact and spherically symmetrical (in particular, $\int_V v\rd v=0$). We define $|V|\bydef\int_V\rd v$. We will consider that $\Omega$ and $V$ are the same for all individuals except in Subsec.~\ref{ssec:confining}, where the confining of \textbf{I} individuals in a one-dimensional spatial model is considered. In this case, $\Omega$ is possibly unbounded and $V_S$, $V_R$, $V_I$ are different discrete sets of admissible velocities for \textbf{S}, \textbf{I}, and \textbf{R} individuals, respectively.

The density $f_j$ evolves according to a Boltzmann-type equation:
\begin{equation}\label{eq:kinetic_basic}
\partial_t f_j + v \cdot \nabla _x f_j = \mathfrak{Q}_j [\mathbf{f}] , \quad \mathbf{f}:= (f_S, f_I , f_R )\ ,\quad  j = S, I, R\ .
\end{equation}
Boundary conditions are given by $ \hat n\cdot\nabla f_j\Big|_{\partial\Omega} = 0$, $j=S,I,R$, where $\hat n$ is the unit exterior normal to the domain $\Omega$.

In the sequel, whenever we refer to results valid to all classes of individuals, we use the subindex $j$, without further specifications. Furthermore, we use the notation $f_j\bydef f_j(t,x,v)$, and $f'_j\bydef f_j(t,x,v')$.

The transition kernel $\mathfrak{Q}_j$ describes three different behaviors for the individual in the population: i) a regression toward the mean change of velocity in the space $V$; ii) a preferred direction in which individuals tends to move towards more favorable environments; and iii) a class-transition kernel accounting to the pathogen spread dynamics, healing and immunity waning, according to the SIRS model.

More specifically, we assume
\[
\mathfrak{Q}_j [\mathbf{f}] = \mathcal{Q}_j ^{(R)} [\mathbf{f}] + \mathcal{Q}_j^{(P)} [\mathbf{f}]+\mathcal{Q}^{(T)}_j[\mathbf{f}]\ , 
\]
where
\begin{itemize}
\item The \emph{regression toward the mean reorientation kernel} is given by
\[
\mathcal{Q}_j ^{(R)} [\mathbf{f}] = \lambda_j (t,x) \left (  \frac{1}{ |V|} \int_V f_j (t,x,v')  \rd v' - f_j(t,x,v) \right )\ ,\quad j=S,I,R\ .
\]
The reorientation rate is given by $\lambda_j >0$. This expression describes a mean reversion process, indicating a drift towards its mean function. In particular, given a point $x\in\Omega$, the number of individuals with velocity $v$ tends to increase if the density function $f_j$ is below average and decrease otherwise.  
We note that the kernel $\mathcal{Q}_j^{(R)}$ is conservative, i.e., $\int_V\mathcal{Q}_j^{(R)}[\mathbf{f}]\rd v=0$. This simple fact will be useful in the subsequent calculations.

\item The \emph{preferred direction} kernel models the rate of change to and from preferable and undesirable directions, respectively. We assume that only susceptible individuals have preferred directions and define
\begin{equation*}
\mathcal{Q}_S ^{(P)}  [\mathbf{f}]\bydef   
\int_V  \bigl( T[f_I] (v,v') f_{S}(t,x,v') - T[f_I] (v',v) f_{S}(t,x,v) \bigr)  \rd v',
\end{equation*}
where the transition from $v'$ to $v$ is given by
\[
T[f_I](v,v') = - \eta(t,x) v \cdot\nabla_x f_I + \xi(t,x) v' \cdot \nabla_x f_I'\ ,
\]
with $\eta, \xi \geq 0$. In particular, this means that susceptible individuals try to avoid directions with increasing concentrations of \textbf{I} individuals (first term in the right-hand side) and to keep the velocity if it points toward a decrease in the concentration of \textbf{I} individuals. For completeness, we also define preferred direction kernels for \textbf{I} and \textbf{R} classes to be identically null:
\[
\mathcal{Q}_I ^{(P)}[\mathbf{f}]\bydef\mathcal{Q}_R ^{(P)}[\mathbf{f}]\bydef 0\ ,\quad \forall\enspace\mathbf{f}\ .
\]
It is straightforward to prove that the kernel $\mathcal{Q}^{(P)}_j$ is also conservative.
\item Finally, we define the {class transition kernel}, where
\[
\mathcal{Q}_j ^{(T)} [\mathbf{f}] = \Psi_j (\mathbf{f}) f_j(t,x,v)\ ,\quad  j=S,I,R\ ,
\]
with
\begin{align*}
\Psi_S (\mathbf{f}) &= - \beta  \frac{f_I}{f_S+f_I+f_R} + \alpha  \frac{f_R}{f_S}\ ,\\
\Psi_I (\mathbf{f}) &= \beta  \frac{f_S}{f_S+f_I+f_R} - \gamma\ ,\\
\Psi_R (\mathbf{f}) &= \gamma \frac{f_I}{f_R} - \alpha\ .
\end{align*}
Parameters $\alpha,\beta,\gamma$ may depend on $t$, $x$, and $v$, except if otherwise stated. We assume that transmission happens when one \textbf{S} and one \textbf{I} individual are aligned in space-phase (i.e., same position and same velocity), which gives more opportunity (i.e., more time) for the pathogen transmission. This is the reason we assume that the class-transition kernel depends on the phase-space density $f_j$ and not on the spatial densities $\int_Vf_j\rd v$. 

Class transitions kernels $\mathcal{Q}^{(T)}_j$ do not satisfy the conservative property.
        \end{itemize}

\begin{remark}
Assuming that $\Psi_j$ represents the relative fitness of type $j$ in a given population, it is clear that the average fitness is zero $\overline{\Psi}=f_S\Psi_S+f_I\Psi_I+f_R\Psi_R=0$. 
The choice of the expression is such that the replicator system of differential equations associated to the fitnesses functions $\Psi_S$, $\Psi_I$, and $\Psi_R$ is given by $f_j'=f_j\left(\Psi_j-\overline{\Psi}\right)$, $\overline{\Psi}=\sum_jf_j\Psi_j$, is the SIRS system of differential equations. For further details, see~\cite{HansenChalub_JTB2024}.
\end{remark}

\section{The macroscopic limit}
\label{sec:limit}

In this section, we study the macroscopic limit of the kinetic model in three steps. In Subsec.~\ref{ssec:admensional} we rewrite Eq.~\eqref{eq:kinetic_basic} in dimensionless form, and identify small parameters; in the sequel, Subsec.~\ref{ssec:mesoscopic}, we identify an advection-reaction-diffusion partial differential equation (PDE) for the evolution of the density of individuals in the domain $\Omega$. We conclude in Subsec.~\ref{ssec:SIRlimit}, that spatially homogeneous solutions of the previous system of PDEs, whenever they exist, correspond to the solution of the SIRS epidemic model~\eqref{eq:SIRS}.

\subsection{Adimensionalization of the kinetic model}
\label{ssec:admensional}

In order to provide the dimensionless form of Eq.~\eqref{eq:kinetic_basic}, we define typical observation time, modulus of displacement, modulus of velocity, and density $t_0$, $x_0$, $v_0$, $f_{0j}$, respectively. See Fig.~\ref{fig:fast_and_slow}.

We define new dimensionless variables $\tilde t\bydef t/t_0$, $\tilde x\bydef x/x_0$, $\tilde v\bydef v/v_0$, and $\tilde f_j\bydef f_j/f_{0j}$ such that
\[
\partial_{\tilde t}\tilde f_j+\frac{t_0v_0}{x_0}\tilde v\cdot\nabla_{\tilde x}\tilde f_j=\frac{t_0}{f_{0j}}\mathfrak{Q}_j[\mathbf{f}_0\tilde{\mathbf{f}}]
\]
We note that $\frac{t_0v_0}{x_0}$ and $\frac{t_0}{f_{j0}}\mathfrak{Q}_j$ are non-dimensional. Furthermore, the microscopic velocity $v_0$ is much larger than the macroscopic velocity $x_0/t_0$, i.e.,
\[
\varepsilon\bydef \left(\frac{v_0}{x_0/t_0}\right)^{-1}\ll1\ .
\]

\begin{figure}[t]
\centering
    \includegraphics[width=0.5\textwidth]{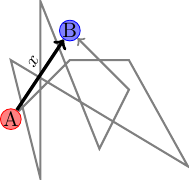}
\caption{Example of a movement of a particle from point A (red) to point B (blue). The particle travels at instantaneous velocity $v$ with trajectory given by the grey line, with typical modulus $v_0$. The direction of movement changes periodically, according to the given kernels. After an observation time $t_0$, the typical displacement (black arrow) is given by $x$, with modulus $x_0$, resulting in a average (macroscopic) velocity of $\frac{x_0}{t_0}\ll v_0$. Typical densities are not represented in the figure, which focus on the movement of a single particle.}
\label{fig:fast_and_slow}
\end{figure}
We assume that typical density is low (an assumption akin to the rarefied gas assumption used in taking macroscopic limits of ideal gases modeled by Boltzmann equations, cf.~\cite{Cercignani_69,Cercignani_88,SaintRaymond_2009} and references therein). More precisely, we redefine the transition kernels $\varepsilon^{-2}\tilde{\mathfrak{Q}}_{j}^{(\varepsilon)}[\mathbf{f}]\bydef \frac{t_0}{f_{0j}}\mathfrak{Q}_j[\mathbf{f}_0\tilde{\mathbf{f}}]$, and omit the $\tilde{}$ notation to find the non-dimensional form of the kinetic equation
\begin{equation}\label{eq:kinetic_epsilon}
\partial_t f_{j} + \frac{1}{\varepsilon} v \cdot \nabla _x f_{j} = \frac{1}{\varepsilon^2}\mathfrak{Q}_{j}^{(\varepsilon)} [\mathbf{f}] ,
\end{equation}
Solution of Eq.~\eqref{eq:kinetic_epsilon} will be denoted by $f_{j}^{(\varepsilon)}$ to stress its dependency on the small positive dimensionless parameter $\varepsilon$.

After proper adimensionalization, the transition kernel will also depend on the parameter $\varepsilon$. As we will see in this section, this implies in particular, that the three different effects discussed in Sec.~\ref {sec:limit} must differ in scales of $\varepsilon$ in order to find the correct macroscopic limit. More precisely, we assume the formal asymptotic of the kernel
\begin{equation}
\label{eq:formalexpansionQ}
 \mathfrak{Q}_{j}^{(\varepsilon)}   = \mathcal{Q}_j^{(R)} + \varepsilon \mathcal{Q}_j ^{(P)}
 + \varepsilon^2 \mathcal{Q}_j ^{(T)}+o\left(\varepsilon^2\right)\ .
\end{equation}

This means, in particular, that the leading order effect in the individual movement is the reorientation, followed by a small-order perturbation that gives preferred directions (when it is non-zero), and a second-order effect guaranteeing that the class transitions are compatible with the SIRS model.

\subsection{The mesoscopic limit}
\label{ssec:mesoscopic}
In this subsection, we derive a partial differential equation for the space and time dependence of the leading order term, i.e., to 
\[
\rho_j\bydef\lim_{\varepsilon\to0}\int_Vf_{j}^{(\varepsilon)}\rd v\ ,\quad (t,x)\in\mathbb{R}_+\times\Omega\ .
\]
The model satisfied by $\rho_j$ will be called \emph{mesoscopic model}, which is an intermediate model between the phase-space evolution described by the kinetic equations in $\mathbb{R}_+\times\Omega\times V$ and the SIRS model in which all variables depend only on $t\in\mathbb{R}_+$. 

Before proceeding to obtain the mesoscopic limit, we remind the reader of a well-known theorem, whose proof can be found in the references:
\begin{theorem}[Lax-Milgram Theorem~\cite{Ramaswamy_1980,Showalter_1997,Bisi_etal}]\label{thm:laxmilgram}
Let $\mathcal{H}$ be a Hilbert space with inner product $\langle\cdot,\cdot\rangle$, and let $\mathcal{B}$ be a continuous bilinear form $\mathcal{B} : \mathcal{H}\times \mathcal{H}\rightarrow \mathbb{R}$. Assume further that $\mathcal{B}$ is  coercive, i.e. $\mathcal{B}(x, x) \geq c \langle x,x\rangle$ for some positive constant $c$. Then, given $w \in \mathcal{H}$, there exists a unique element $x \in \mathcal{H}$ such that $ \mathcal{B}(x, u) = \langle w, u \rangle$ for all $u \in \mathcal{H}$.
\end{theorem}

In the sequel, we will use a more specific form of the Lax-Milgram theorem:

\begin{corollary}\label{cor:laxmilgram}
Let $\mathcal{H}_0=\{h\in L^2(V)|\int_{V}h\rd v=0\}$. Then, for each $\mathbf{g}=(g_S,g_I,g_R)$, with $g_j\in\mathcal{H}_0$ there is a unique $\mathbf{f}=(f_S,f_I,f_R)$ such that $f_j\in\mathcal{H}_0$,  and $\mathcal{Q}^{(R)}_j[\mathbf{f}]=-\mathbf{g}$, $j=S,I,R$, in the weak sense.
\end{corollary}

\begin{proof}
We start by defining 
\[
W_\lambda[f]\bydef \lambda\left(\frac{1}{|V|}\int_Vf'\rd v'-f\right)\ ,
\]
for strictly positive $\lambda$.
It is clear that $\mathcal{H}_0$ is a Hilbert space with the natural $L^2$-inner product, $\langle f,g\rangle\bydef\int_V fg\rd v$. Define the bilinear form 
$\mathcal{B}(g,h)\bydef-\int_V W_\lambda[g]h\rd v$, and 
\[
\mathcal{B}(g,g)=-\lambda
\left(\int_{V\times V}\frac{1}{|V|}g'g\rd v'\rd v-\int_Vg^2\rd v\right)=\lambda\langle g,g\rangle.
\]
We conclude that $\mathcal{B}$ is coercive. The continuity is immediate.

The Lax-Miligram Theorem~\ref{thm:laxmilgram} guarantees that for each $g\in\mathcal{H}_0$, there is a unique $f\in\mathcal{H}_0$ such that
\[
\langle -g,h\rangle=-\mathcal{B}(f,h)=\int_V W_\lambda[f]h\rd v\ ,\quad\forall h\in\mathcal{H}_0\ .
\]
Noting that $\mathcal{Q}^{(R)}_j[\mathbf{f}]=W_{\lambda_j}[f_j]$, we finish the proof.
\end{proof}

We assume the formal expansion of the solution
\[
 f_j^{(\varepsilon)}=f_{j,0}+\varepsilon f_{j,1}+\varepsilon^2f_{j,2}+o\left(\varepsilon^2\right)\ ,
\]
 where $f_{j,i}$ is the $i$-th order expansion in $\varepsilon$ of $f_j^{(\varepsilon)}$.

The objective of this subsection is to find a system of PDEs for $\rho_j=\int_Vf_{j,0}\rd v$.
From Eq.~\eqref{eq:formalexpansionQ}, it is clear that $\mathcal{Q}^{(R)}_j$, $\mathcal{Q}^{(P)}_j$, and $\mathcal{Q}^{(T)}_j$ should be expanded up to second, first, and zero order in $\varepsilon$, respectively. Ommitting terms of higher order, we find:
\begin{align*}
\mathcal{Q}^{(R)}_j[\mathbf{f}]&=\lambda_j\left[\frac{1}{V}\int_V \left(f_{j,0}'+\varepsilon f_{j,1}'+\varepsilon^2 f_{j,2}'\right)\rd v'-\left(f_{j,0}+\varepsilon f_{j,1}+\varepsilon^2 f_{j,2}\right)\right]\\
&=\mathcal{Q}_j^{(R)}[\mathbf{f}_{0}]+\varepsilon\mathcal{Q}_j^{(R)}[\mathbf{f}_{1}]+\varepsilon^2\mathcal{Q}_j^{(R)}[\mathbf{f}_{2}]\ ,\\
\mathcal{Q}^{(P)}_S[\mathbf{f}]&=\int_V\left\{T\left[f_{I,0}+\varepsilon f_{I,1}\right](v,v')\left(f_{S,0}'+\varepsilon f_{S,1}'\right)\right. \\
&\quad - \left.T\left [f_{I,0}+\varepsilon f_{I,1}\right](v',v)\left(f_{S,0}+\varepsilon f_{S,1}\right)\right\}\rd v'\\
&=\int_V\left(T[f_{I,0}](v,v')f_{S,0}'-T[f_{I,0}](v',v)f_{S,0}\right)\rd v'+\\
&\quad+\varepsilon\left\{\int_V\left(T[f_{I,0}](v,v')f_{S,1}'-T[f_{I,0}](v',v)f_{S,1}\right)\rd v'\right.\\
&\qquad\left.+\int_V\left(T[f_{I,1}](v,v')f_{S,0}'-T[f_{I,1}](v',v)f_{S,0}\right)\rd v'\right\}\\
&=\mathcal{Q}^{(P)}_S[\mathbf{f}_{0}]+\overline{\mathcal{Q}}^{(P)}_S[\mathbf{f}_{0},\mathbf{f}_{1}]\ ,\\
\end{align*}
where 
\begin{align*}
    \overline{\mathcal{Q}}^{(P)}_S[\mathbf{f}_{0},\mathbf{f}_{1}]&\bydef
  \int_V\left(T[f_{I,0}](v,v')f_{S,1}'-T[f_{I,0}](v',v)f_{S,1}\right)\rd v'\\
  &\quad+\int_V\left(T[f_{I,1}](v,v')f_{S,0}'-T[f_{I,1}](v',v)f_{S,0}\right)\rd v'
\end{align*}
is conservative and $\mathbf{f}_{\kappa}\bydef (f_{S,\kappa},f_{I,\kappa},f_{R,\kappa})$, $\kappa=0,1$.

Finally, expanding both sides of Eq.~\eqref{eq:kinetic_epsilon}, we find
\begin{align*}
       &\varepsilon^2  \partial _t f_{j,0 } + \varepsilon  v\cdot \nabla _x f_{j,0 } + \varepsilon^2  v \cdot \nabla _x f_{j,1 } \\
       &\quad=   \mathcal{Q}_j ^{(R)} [ \mathbf{f}_{0}] + \varepsilon  \left ( \mathcal{Q}_j^{(R)}[ \mathbf{f}_{1}]  + \mathcal{Q}_j^{(P)}[ \mathbf{f}_{0}] \right ) \\
        &\qquad+ \varepsilon^2 \left ( \mathcal{Q}_j ^{(R)}[ \mathbf{f}_{2}]  + \overline{\mathcal{Q}}_j^{(P)}[\mathbf{f}_0, \mathbf{f}_{1}] +\mathcal{Q}_j ^{(T)}[ \mathbf{f}_{0}] \right )\ .
    \end{align*}
    where $\overline{\mathcal{Q}}_{j}^{(P)}\equiv 0$ for $j=I,R$.

In the remainder of this subsection, we compare terms of the same order of magnitude on both sides of the previous equation. 

\subsubsection{Leading Order Equation}
Equating terms of order $\varepsilon^0=1$, we find
\[
0 = \mathcal{Q}_j ^{(R)} [ \mathbf{f}_{0}] = \lambda_j \left ( \frac{1}{ |V|} \int_V f_{j,0} '   \rd v' - f_{j,0} \right )\ .
\]
The solution of the previous equation is such that
\[
    f_{j,0} =  \frac{1}{ |V|} \int_V f_{j,0} '  \rd v'\ .
\]
We conclude that for all $j$, $\int_ V v f_{j,0} \rd v = 0$, and this implies, in particular, that $f_{j,0}$ is spherically symmetric, i.e., there is a function $F:V\to\mathbb{R}_+$, with $F(v)=F(|v|)$, such that 
\begin{equation}\label{eq:leadingorder}
f_{j,0}(t,x,v)=\rho_j(x,t)F(|v|)\ ,
\end{equation}
and $\int_VF(v)\rd v=1$.

\subsubsection{Order 1 Equation}
Equating terms of order $\varepsilon^1=\varepsilon$, we find
\begin{equation*}
    v \cdot \nabla _x f_{j, 0 } = \mathcal{Q}_j ^{(R)}[ \mathbf{f}_{1}]  + \mathcal{Q}_j ^{(P)}[ \mathbf{f}_{0}]
\end{equation*}

For $j= I,R$, we only need to solve the equation
\begin{equation}\label{eq:order_one}
 \lambda_j  \left( \frac{1}{ |V|} \int_V f_{j,1} ' \rd v' - f_{j,1} \right)=\mathcal{Q}_j ^{(R)} [ \mathbf{f}_{1}] = Fv \cdot
  \nabla_x \rho_{j}\ .
\end{equation}
The solution is given by
\begin{equation}\label{eq:fIR1}
f_{j,1} = - \kappa_j (t,x,v) \nabla_x \rho_{j}(t,x) + \chi_{j}(t,x) F(v)\ .
\end{equation}
More precisely, from Corolary~\ref{cor:laxmilgram}, it follows that there is $ \bm{\kappa} (t,x,v)$ solution of 
\begin{equation}\label{eq:kappa}
\mathcal{Q}_j ^{(R)} [ \bm{\kappa}] = -vF\ ,\quad\bm{\kappa}=(\kappa_S,\kappa_I,\kappa_R)\ ,\quad j=S,I,R\ ,
\end{equation}
and $\chi_j F$ is in the kernel of the operator $\mathcal{Q}_j^{(R)}$.

For $j=S$, we need to add an extra term on the right-hand side of the Eq.~\eqref{eq:order_one}:
\begin{align*}
 \mathcal{Q}_S^{(P)}[\mathbf{f}_0]&=\int_V \Big(- \eta (t,x)  v \cdot (\nabla _x f_ {I,0}) f_{S,0} ' + \xi (t,x)  v' \cdot (\nabla_x f_ {I,0}') f_{S,0} '\\
        &\quad+ \eta (t,x)  v' \cdot  (\nabla _x f_ {I,0}')  f_{S,0} - \xi (t,x)  v \cdot (\nabla_x f_ {I,0} )f_{S,0}  \Big) \rd v'\\
         &= \int_V\Big( -\eta FF'  v \cdot    (\nabla _x \rho_{I} )  \rho_{S} + \xi F'^2  v' \cdot  (\nabla_x \rho_{I} ) \rho_{S} \\
         &\quad+ \eta FF' v' \cdot (\nabla _x  \rho_{I} ) \rho_{S} - \xi F^2 v \cdot (\nabla_x \rho_{I} ) \rho_{S} \Big)\rd v'\\
         & = - ( \eta  + \xi F|V| ) vF \cdot (\nabla _x  \rho_{I} ) \rho_{S}\ ,
\end{align*}
where $\eta=\eta(t,x)$, $\xi=\xi(t,x)$, and $F'=F(v')=F(|v'|)$. Note that one particular realisation of the leading order equilibrium distribution $F$ is $F(v)=\frac{1}{|V|}$. This will be the case whenever $V$ is a hypersphere (a circle, in the two-dimensional case). In this case, the last equation simplifies to $ \mathcal{Q}_S^{(P)}[\mathbf{f}_0]=- ( \eta  + \xi ) vF \cdot (\nabla _x  \rho_{I} ) \rho_{S}$.

We conclude that
\begin{equation}\label{eq:fS1}
f_{S,1} =  - \Theta(t,x,v) ( \nabla_x \rho_{I} )
\rho_{S} - \kappa_S(t,x,v) \nabla_x \rho_{S} + \chi_{S} F
\end{equation}
where $\Theta (t,x,v)$ is the solution of $ \mathcal{Q}_S^{(R)} [ \Theta] = - ( \eta  + \xi )  vF$, cf. Corollary~\ref{cor:laxmilgram}.

\subsubsection{Order 2 equations}
\label{sssec:order2}

The order $\varepsilon^2$ equation is given by
 \begin{equation*}
    \partial_t f_{j,0 } + v \cdot \nabla_x f_{j, 1 } = \mathcal{Q}_j^{(R)}[ \mathbf{f}_{2}]  + \overline{\mathcal{Q}}^{(P)}_j[\mathbf{f}_0, \mathbf{f}_{1}] +\mathcal{Q}_j^{(T)}[ \mathbf{f}_{0}]\ .
\end{equation*}

Integrating on $V$ and using Eqs.~\eqref{eq:leadingorder} using the conservative property of $\overline{\mathcal{Q}}^{(P)}_j$, we find
\[
 \partial_t\rho_{j,0}+\int_V v\cdot\nabla_xf_{j,1}\rd v=\int_V\mathcal{Q}_j^{(T)}[\mathbf{f_0}]\rd v\ .
\]

For $j=I$, we find
\begin{align*}
& \partial_t\rho_{I}-\int_V v\cdot\nabla_x\left[\kappa_{I}(t,x,v)\nabla_x\rho_{I}\right]\rd v+\int_Vv\cdot\nabla_x\left[\chi_{I}F(v)\right]\rd v\\
&\qquad=\int_V\left(\beta\frac{\rho_{S}}{\rho_{S}+\rho_{I}+\rho_{R}}-\gamma\right)\rho_{I}F(v)\rd v=\beta_0\frac{\rho_{S}\rho_{I}}{\rho_{S}+\rho_{I}+\rho_{R}}-\gamma_0\rho_{I}\ ,
\end{align*}
where $\beta_0\bydef\int_V\beta F(v)\rd v$, $\gamma_0\bydef\int_V\gamma F(v)\rd v$; we define a similar equation for the loss of immunity parameter: $\alpha_0\bydef\int_V\alpha F(v)\rd v$. Due to the normalization $\int_VF(v)\rd v=1$, this means that, for each $x$ and $t$, $\alpha_0$, $\beta_0$, and $\gamma_0$ are averages at equilibrium of the corresponding parameters in the $V$-space.

Note that $\int_Vv\cdot\nabla_x\left[\rho_{I}F(v)\right]\rd v=
\int_VvF(v)\rd v\cdot \nabla_x\rho_{I}=0$. 

Defining $D(t,x)\bydef\int_Vv\cdot\kappa(t,x,v)\rd v$, we conclude
\begin{equation}\label{eq:FPI}
 \partial_t\rho_{I}=\nabla_x\left[D(t,x)\nabla_x\rho_{I}\right]+\beta_0\frac{\rho_{S}\rho_{I}}{\rho_{S}+\rho_{I}+\rho_{R}}-\gamma_0\rho_{I}\ .
\end{equation}

The equation for $j=R$ is similar:
\begin{equation}\label{eq:FPR}
 \partial_t\rho_{R}=\nabla_x\left[D(t,x)\nabla_x\rho_{R}\right]+\int_V\mathcal{Q}_j^{(T)}[\mathbf{f}_0]\rd v=\nabla_x\left[D(t,x)\nabla_x\rho_{R}\right]+\gamma_0\rho_{I}-\alpha_0\rho_{R}\ .
\end{equation}

Finally, we study the equation for $j=S$. The only difference from the previous case is the existence of an extra term in Eq.~\eqref{eq:fS1} when compared to Eq.~\eqref{eq:fIR1}:
\[
 \int_Vv\cdot\nabla_x\left[\Theta(t,x,v)(\nabla_x\rho_{I})\rho_{S}\right]\rd v=\nabla_x\cdot\left[\Gamma(t,x)(\nabla_x\rho_{I})\rho_{S}\right]\ ,
\]
where $\Gamma(t,x)=\int_Vv\cdot\Theta(t,x,v)\rd v$. Finally,
\begin{equation}\label{eq:FPS}
 \partial_t\rho_{S}=\nabla_x\left[D(t,x)\nabla_x\rho_{S}+\Gamma(t,x)\left(\nabla_x\rho_{I}\right)\rho_{S}\right]+\alpha_0\rho_{I}-\beta_0\frac{\rho_{S}\rho_{I}}{\rho_{S}+\rho_{I}+\rho_{R}}\ .
\end{equation}

\begin{remark}
    Assume that the transmission and recovery rates are such that an outbreak is locally prevented for any point in $\Omega$, for any time $t$, i.e., $\beta_0<\gamma_0$ for $(t,x)\in\mathbb{R}_+\times\Omega$. Multiplying Eq.~\eqref{eq:FPI} by $\rho_I$ and integrating in $\Omega$, we find
    \[
    \frac{1}{2}\frac{\rd\ }{\rd t}\int_{\Omega}\rho_I^2\rd x\le-\int_\Omega D(t,x)|\nabla_x\rho_I|^2\rd x-\int_\Omega(\gamma_0-\beta_0)\rho_I^2\rd x<0\ ,
    \]
    where we used that $\frac{\rho_S}{\rho_S+\rho_I+\rho_R}\le 1$. We conclude that $\rho_I\to0$ for all $x\in\Omega$ when $t\to\infty$, i.e., the system approaches asymptotically the disease-free state. The condition on $\beta_0$ and $\gamma_0$ could be weakened without losing the disease-free state's uniqueness as the only globally asymptotic attractor of the dynamics.
\end{remark}

\subsection{The SIRS epidemic model}
\label{ssec:SIRlimit}

Let us assume a space homogeneous problem, i.e., we consider that $\alpha_0$, $\beta_0$, and $\gamma_0$ do not depend on $x$, and that the initial condition is also space independent. Therefore, there is
a space homogeneous solutions of the system~\eqref{eq:FPI}--\eqref{eq:FPS}, such that $(\rho_S,\rho_I,\rho_R)$ is the solution of SIRS model~\ref{eq:SIRS}, with $S=\rho_{S}$, $I=\rho_{I}$, and $R=\rho_{R}$.

However, when working with the SIRS model, we are interested in the total number of type \textbf{S}, \textbf{I}, and \textbf{R} individuals (possibly normalised to the total population), and not its average space-density. Therefore, it is natural to define
\[
S=\int_\Omega\rho_S\rd x\ ,\quad
I=\int_{\Omega}\rho_I\rd x\ ,\quad
R=\int_{\Omega}\rho_R\rd x\ .
\]
We find that $(S,I,R)$ is a solution of the SIRS model for constant-in-space $\rho_j$ with parameters
\begin{align*}
\bar\alpha&\bydef\frac{1}{|\Omega|}\int_{\Omega}\alpha_0 \rd x=\frac{1}{|\Omega|}\int_{\Omega\times V}\alpha F(v)\rd v\rd x=\alpha_0\ ,\\
\bar\beta&\bydef\frac{1}{|\Omega|}\int_{\Omega}\beta_0\rd x=\frac{1}{|\Omega|}\int_{\Omega\times V}\beta F(v)\rd v\rd x=\beta_0\ ,\\
\bar\gamma&\bydef\frac{1}{|\Omega|}\int_{\Omega}\gamma_0\rd x=\frac{1}{|\Omega|}\int_{\Omega\times V}\gamma F(v)\rd v\rd x=\gamma_0\ .
\end{align*}
We finish this subsection, noting that it is not difficult to generalise all the above processes for more general class transition kernels. In particular, it is possible to choose class transition kernels 
 \[
 \mathcal{Q}_j^{(T)}[\mathbf{f}]=f_j\left(\Psi_j[\mathbf{f}]-
 \sum_jf_j\Psi_j[\mathbf{f}]\right)\ ,
 \]
 where $\Psi_j[\mathbf{f}]$ is the \emph{fitness} (reproductive viability) function for a population at state $\mathbf{f}$ (i.e, the presence of the type $j$ individuals is given by $f_j$. In this case, $\dot f_j=\mathcal{Q}_j^{(T)}[\mathbf{f}]$ is the well-known replicator equation~\cite{HofbauerSigmund}. 

 From the discussion in this section, the kinetic formulation is one possible generalisation of the replicator equation, including not only space-inhomogeneity but also other variables (in particular, in this work, velocity).

\section{Existence and uniqueness of solutions of the kinetic model}
\label{sec:existence}

In this section, we prove the existence of solutions of the kinetic model~\eqref{eq:kinetic_epsilon}, for fixed $\varepsilon>0$ (that will be omitted along this section), under the assumption that there is no preferred direction of movement, i.e., $\eta=\xi=0$.  This assumption follows from the technical difficulty to prove existence and uniqueness results in kinetic modeling when transition kernels depend on the gradient of solutions, cf.~\cite[Remark 4]{Chalubetal_04}; see also~\cite{HillenStevens_00} for the one-dimensional case with dependence on the gradient. 

\newcommand{\ini}{{\mathrm{ini}}}

\begin{theorem}\label{thm:existence}
    Let  $\alpha,\beta,\gamma\le C$, and consider $f^\ini_j\in L^p(\Omega\times V)$, for $p\ge 2$. Assume $\mathcal{Q}^{(P)}_j\equiv 0$, $j=S,I,R$. Then, for $\varepsilon>0$, there exists a global solution $(f_S,f_I,f_R)\in\left[L^\infty(\mathbb{R}_+;L^p(\Omega\times V))\right]^3$ of the kinetic system~\eqref{eq:kinetic_epsilon}, with initial condition given by $(f_S^\ini,f_I^\ini,f_R^\ini)$.
\end{theorem}

\begin{proof}
We assume without loss of generality that $\varepsilon=1$. Whenever we refer to a constant $C$ (with or without index), we consider that ``there exists a positive number $C$ such that''. Its precise value is immaterial to the proof of the theorem, and the same parameter $C$ may refer to different constants along the proof.

We multiply Eq.~\eqref{eq:kinetic_epsilon} by $f_j^{p-1}$ and integrate over $\Omega\times V$:
\begin{align}\label{eq:Lpnorm_partial}
\frac{1}{p}\frac{\rd\ }{\rd t}\int_{\Omega\times V} f_j^{p}\rd v\rd x&=
\int_{\Omega\times V} f_j^{p-1}\left[\partial_tf_j+v\cdot\nabla_xf_j\right]\rd v\rd x\\
\nonumber
&=\int_{\Omega\times V} f_j^{p-1}\mathfrak{Q}_j[\mathbf{f}]\rd v\rd x
\end{align}
We remember the Holder's inequality: 
\[
\int |gh|\le\left(\int |g|^{q_1}\right)^{1/q_1}\left(\int |h|^{q_2}\right)^{1/q_2}\ ,\quad \frac{1}{q_1}+\frac{1}{q_2}=1\ ,\quad q_1,q_2\ge 1\ .
\]

For $g=f_j\ge 0$ and $h=1$, with $q_1=p$, $q_2=\sfrac{p}{(p-1)}$ we conclude that
\begin{equation}\label{eq:inequalityA}
\int_{V}f_j\rd v\le|V|^{1-\frac{1}{p}}\left(\int_Vf_j^{p}\rd v\right)^{1/p}\ .
\end{equation}
For $g=f_j^{p-1}$, $h=1$, $q_1=\sfrac{p}{(p-1)}$, $q_2=p$, we find
\begin{equation}\label{eq:inequalityB}
\int_Vf_j^{p-1}\rd v\le |V|^{\frac{1}{p}}\left(\int_Vf_j^p\rd v\right)^{1-\frac{1}{p}}\ .
\end{equation}
We study each term in the right-hand-side of Eq.~\eqref{eq:Lpnorm_partial}, using that $\mathfrak{Q}_j=\mathcal{Q}^{(R)}_j+\mathcal{Q}^{(T)}_j$. For $\mathcal{Q}^{(R)}_j$, we find
\begin{align}
\nonumber
     \int_{\Omega\times V} f_j^{p-1}Q_j^{(R)}[\mathbf{f}]\rd v\rd x&=\int_\Omega\lambda_j\left[\frac{1}{|V|}\int_{V\times V}f_j'f_j^{p-1}\rd v'\rd v-\int_Vf_j^p\rd v\right]\rd x\\
\nonumber
&=\int_\Omega\lambda_j\left[\frac{1}{|V|}\int_{V}f_j\rd v\int_Vf_j^{p-1}\rd v-\int_Vf_j^p\rd v\right]\rd x\\
\label{eq:inequalityQR}
&\le 0\ ,
\end{align}
where the last inequality follows from the multiplication of Inequalities~\eqref{eq:inequalityA} and~\eqref{eq:inequalityB}. For $\mathcal{Q}^{(T)}_I$, we find
\begin{align}
\nonumber
\int_{\Omega\times V}f_I^{p-1}\mathcal{Q}^{(T)}_I[\mathbf{f}]\rd v\rd x&=
\int_{\Omega\times V}\left[\frac{\beta f_Sf_I^p }{f_S+f_I+f_R}-\gamma f_I^p\right]\rd v\rd x\\
\label{eq:inequalityQTI}
&\le C\|f_I\|_{L^p(\Omega\times V)}^p\ ,
\end{align}
where we used that $\sfrac{f_S}{(f_S+f_I+f_R)}\le 1$.
Therefore, from Eqs.~\eqref{eq:Lpnorm_partial}, \eqref{eq:inequalityQR}, with $j=I$, and Eq.~\eqref{eq:inequalityQTI}, we find
\[
\frac{\rd\ }{\rd t}\|f_I\|_{L^p(\Omega\times V )}\le C\|f_I\|_{L^p(\Omega\times V)}\ \quad\forall t . 
\]
Using Gronwall's inequality~\cite{Grownall_1919}, we conclude that $f_I\in L^\infty(\mathbb{R}_+;L^p(\Omega\times V))$. 

Using Holder's inequality, with $g=f_I$, $h=f_R^{p-1}$, $q_1=p$, $q_2=\sfrac{p}{(p-1)}$, we show that
\begin{align}
\nonumber
\int_{\Omega\times V}f_R^{p-1}\mathcal{Q}^{(T)}_R[\mathbf{f}]\rd v\rd x&=
\int_{\Omega\times V}\left[\gamma f_If_R^{p-1}-\alpha f_R^p\right]\rd v\rd x\\
\nonumber
&\le\int_{\Omega\times V}\gamma f_If_R^{p-1}\rd v\,\rd x\\
\label{eq:inequalityQTR}
&\le C\|f_I\|_{L^p(\Omega\times V)}\|f_R\|_{L^p(\Omega\times V)}^{p-1}\ .
\end{align}
Using Eqs.~\eqref{eq:Lpnorm_partial},~\eqref{eq:inequalityQR}, ~\eqref{eq:inequalityQTR}, and the existence result for $f_I$, we conclude
\[
\|f_R\|^{p-1}_{L^p(\Omega\times V)}\frac{\rd\ }{\rd t}\|f_R\|_{L^p(\Omega\times V)}=
\frac{1}{p}\frac{\rd\ }{\rd t}\|f_R\|^p_{L^p(\Omega\times V)}\le C\|f_R\|^{p-1}_{L^p(\Omega\times V)}\ ,
\]
and, therefore,
\[
\frac{\rd\ }{\rd t}\|f_R\|_{L^p(\Omega\times V)}\le C\ ,\quad\forall t.
\]
We conclude that $f_R\in L^\infty(\mathbb{R}_+;L^p(\Omega\times V))$.

We proceed similarly to prove that
\[
\int_{\Omega\times V}f_S^{p-1}\mathcal{Q}^{(T)}_S[\mathbf{f}]\rd v\rd x\le C\|f_S\|_{L^p(\Omega\times V)}^{p-1}\ ,
\]
and, therefore, $f_S\in L^\infty(\mathbb{R}_+;L^p(\Omega\times V))$.
\end{proof}

Now, we prove the uniqueness of the solution. 

\begin{theorem}
\label{thm:uniqueness}
    Consider the same conditions of Thm.~\ref{thm:existence}, with $p=2$. Then the solution given by Thm.~\ref{thm:existence} is unique.
\end{theorem}

\begin{proof}
We assume again that $C$ is an arbitrary, possibly distinct, constant.
Let us define the norm
\begin{equation}\label{eq:normf}
\|\mathbf{f}-\mathbf{g}\|^2\bydef\int_{\Omega\times V}\left[ (f_S-g_S)^2+(f_I-g_I)^2+(f_R-g_R)^2\right]\rd v\,\rd x\ .
\end{equation}

From the linearity of $\mathcal{Q}^{(R)}_j$ it is not difficult to prove that
\[
\mathcal{Q}^{(R)}_j[\mathbf{f}]-\mathcal{Q}^{(R)}_j[\mathbf{g}]=\mathcal{Q}^{(R)}_j[\mathbf{f}-\mathbf{g}]\ .
\]
Therefore, from Eq.~\eqref{eq:inequalityQR},
\begin{align}\nonumber
    \int_{\Omega\times V}\left(\mathcal{Q}^{(R)}_j[\mathbf{f}]-\mathcal{Q}^{(R)}_j[\mathbf{g}]\right)\left(f_j-g_j\right)\rd v\rd x&=\int_{\Omega\times V}\mathcal{Q}^{(R)}_j[\mathbf{f}-\mathbf{g}](f_j-g_j)\rd v\,\rd x\\
    \label{eq:uniquenessQR}
    &\le 0\ .
\end{align}

We study the class-transition kernel:
\begin{align*}
&\left(\mathcal{Q}^{(T)}_S[\mathbf{f}]-\mathcal{Q}^{(T)}_S[\mathbf{g}]\right)\left(f_S-g_S\right)+
    \left(\mathcal{Q}^{(T)}_I[\mathbf{f}]-\mathcal{Q}^{(T)}_I[\mathbf{g}]\right)\left(f_I-g_I\right)\\
    &\qquad+
    \left(\mathcal{Q}^{(T)}_R[\mathbf{f}]-\mathcal{Q}^{(T)}_R[\mathbf{g}]\right)\left(f_R-g_R\right)\\
    &\quad=\left(-\beta\frac{f_If_S}{f_S+f_I+f_R}+\alpha f_R+\beta\frac{g_Ig_S}{g_S+g_I+g_R}-\alpha g_R\right)(f_S-g_S)\\
    &\qquad+\left(\beta\frac{f_If_S}{f_S+f_I+f_R}-\gamma f_I-\beta\frac{g_Ig_S}{g_S+g_I+g_R}+\gamma g_R\right)(f_I-g_I)\\
    &\qquad+\left(\gamma f_I-\alpha f_R-\gamma g_I+\alpha g_R\right)(f_R-g_R)\\
    &\quad=\beta\left(\frac{f_If_S}{f_S+f_I+f_R}-\frac{g_Ig_S}{g_S+g_I+g_R}\right)\left((f_I-g_I)-(f_S-g_S)\right)\\
    &\qquad+\alpha(f_R-g_R)(f_S-g_S)-\gamma(f_I-g_I)^2+\gamma(f_I-g_I)(f_R-g_R)\\
    &\qquad-\alpha(f_R-g_R)^2\ .
\end{align*}
To simplify the notation, we define $A_f\bydef f_S+f_I+f_R$, and a similar notation for $A_g$. We prove that
\begin{align*}
    \left|\frac{f_If_S}{A_f}-\frac{g_Ig_S}{A_g}\right|&=\frac{1}{A_fA_g}\left|f_Sg_S(f_I-g_I)+f_Ig_I(f_S-g_S)+f_Sg_R(f_I-g_I)\right.\\
    &\quad\left.-f_Sg_I(f_R-g_R)+f_Rg_I(f_S-g_S)\right|\\
    &\le C\left(|f_S-g_S|+|f_I-g_I|+|f_R-g_R|\right)\ .
\end{align*}
In the last passage, we used that
\[
\frac{f_jg_{j'}}{A_fA_g}\le 1\ .
\]

Using the inequality for $\mathcal{Q}^{(T)}_j$ and the classical inequality $ab\le\frac{1}{2}(a^2+b^2)$, we prove
\begin{align*}
&\sum_{j=S,I,R}\left(\mathcal{Q}^{(T)}_j[\mathbf{f}]-\mathcal{Q}^{(T)}_j[\mathbf{g}]\right)\left(f_j-g_j\right)\\   
&\quad\le C_1(f_S-g_S)^2+C_2(f_I-g_I)^2+C_3(f_R-g_R)^2\ .
\end{align*}
Finally, we conclude that
\begin{equation}
\label{eq:uniquenessQT}
\int_{\Omega\times V}\sum_{j=S,I,R}\left(\mathcal{Q}^{(T)}_j[\mathbf{f}]-\mathcal{Q}^{(T)}_j[\mathbf{g}]\right)\left(f_j-g_j\right)\rd v\,\rd x   
\le C\|\mathbf{f}-\mathbf{g}\|^2\ .
\end{equation}
Finally, from Eqs.~\eqref{eq:uniquenessQR},~\eqref{eq:uniquenessQT}, and the definition of $\mathfrak{Q}_j$, we conclude that
\begin{align*}
 \frac{1}{2}\frac{\rd\ }{\rd t}\|\mathbf{f}-\mathbf{g}\|^2&=\sum_j\int_{\Omega\times V}(f_j-g_j)\left(\partial_t(f_j-g_j)+v\cdot\nabla_x(f_j-g_j)\right)\rd v\,\rd x\\
 &=\sum_j\int_{\Omega\times V}\mathfrak{Q}_j[\mathbf{f}-\mathbf{g}](f_j-g_j)\rd v\,\rd x\\
 &\le C\sum_j\int_{\Omega\times V}(f_j-g_j)^2\rd v\rd x = C\|\mathbf{f}-\mathbf{g}\|^2\ ,
\end{align*}
and the uniqueness follows from an standard application of Gronwall's inequality for initial conditions $\mathbf{f}^\ini=\mathbf{g}^\ini$.
\end{proof}

\begin{remark}
    As discussed at the beginning of this section, we did not consider the preferred kernel in the proof of Thms.~\ref{thm:existence} and~\ref{thm:uniqueness}. As will be discussed in the conclusions, we believe that the existence and uniqueness theorem still holds with parameters $\eta,\xi>0$. 
\end{remark}

We finish noting that the existence and uniqueness of solutions of the mesoscopic model~\eqref{eq:FPI}--\eqref{eq:FPS}, when $\Gamma\equiv0$, can be proved in a similar way. However, we will not consider the case $\Gamma$ not identically zero, i.e., when \textbf{S} individuals actively try to escape from \textbf{I} individuals.    

\section{Examples and applications}
\label{sec:examples}

The primary motivation for incorporating kinetic modelling into epidemiology is to elucidate how variations in population motility influence disease transmission~\cite{Bertagliaetal_21B}. This approach is especially pertinent when policymakers evaluate large-scale confinement strategies. Comprehensive reviews of the effects of widespread prophylactic measures on disease dynamics are available in~\cite{Lawrence2024,Khumbudzo2025}.

A systematic analysis of the effects of confining \textbf{I} individuals would require considering class-dependent admissible velocity sets $V$, which extend beyond the current theoretical framework. Nevertheless, the effects of confining \textbf{I} individuals in one-dimensional cases will be formally addressed in Subsec.~\ref{ssec:confining}.

Before that, we will consider three other examples of kinetic models that are consistent with the SIRS epidemic model. In Subsec.~\ref{ssec:numerical}, we will simulate the kinetic models numerically to illustrate how individuals disperse within the discretised mesoscopic model. 

The second example, which will be developed in Subsec.~\ref{ssec:constant}, is pedagogical in nature and deals with a situation involving constant parameters, in order to clearly demonstrate the relationship between the kinetic model and the standard (space-homogeneous) SIRS dynamics. This is also important for understanding further expansions of our theory to include more general interaction kernels, which will be discussed in the conclusions.

The third example, in Subsec.~\ref{ssec:skeptical}, considers a situation that unfortunately occurs in several diseases, whereby large parts of the population in certain localities, ranging from small villages to entire countries, behave as if the disease did not exist. This is frequently driven by misinformation campaigns. However, outside this region, susceptible individuals naturally try to avoid infectious individuals.

Finally, Subsec.~\ref{ssec:confining} considers infectious individuals voluntarily confining themselves, which is represented by a decrease in the velocity of \textbf{I} individuals with respect to \textbf{S} and \textbf{R} individuals, as discussed at the beginning of this section. We demonstrate that confining \textbf{I} individuals reduces the intensity of disease spread.

\subsection{A numerical exploration}
\label{ssec:numerical}

Consider a large number of particles moving within a finite-size square lattice with periodic boundary conditions. Individuals are of type \textbf{S}, \textbf{I}, and \textbf{R}. $S(x,v)$ is the number of individuals of type \textbf{S} at $x$ moving at velocity $v$, with similar notation for $I(x,v)$ and $R(x,v)$. 

Initially, all individuals are uniformly distributed with random initial velocities in a central square whose side is one-third the size of the full lattice. The majority are of type \textbf{S}, with a small number of type \textbf{I} individuals. No \textbf{R} individuals are initially present. 

The update consists of two sequential steps: first, the moving direction is updated and the particle moves accordingly; then, class transitions are considered.

We start by moving each particle. For each individual, we consider two components of movement: the first in the horizontal direction and the second in the vertical direction. These are updated independently, resulting in eight admissible velocities corresponding to the set V: north, northwest, west, southwest, south, southeast, east and northeast. Consider an individual moving in one of the two main directions: horizontal or vertical. If the number of individuals moving in the same direction exceeds the average, there is a probability of changing their moving direction, denoted by a value of $\lambda\in(0,1]$. Otherwise, the direction remains unchanged. After performing the calculations for the horizontal and vertical directions independently, we update that individual's moving direction. For example, if an individual at position $x$ with velocity $v$ moves northeast ($v'= (1, 1)$), then we decrease $I(x, v)$ by one and increase $I(x + (1, 1), (1, 1))$ by one. All updates are performed only after running the update rule for all particles (synchronous update).

We did not consider preferred direction kernels, i.e. we set all of the $\mathcal{Q}^{(P)}_j$ kernels to zero.

Once all the movements have been performed, the class transitions are considered. The parameters $\alpha$, $\beta$, and $\gamma$ are fixed. For each individual, a random number $r$ is drawn uniformly between 0 and 1. If the individual is of class \textbf{S} and $r_1 < \beta I(x, v) / (S(x, v) + I(x, v) + R(x, v))$, then $S(x, v)$ decreases by one and $I(x, v)$ increases by one, indicating the \textbf{S} to \textbf{I} transition. If the individual is of class \textbf{I} and $r < \gamma$, then $I(x, v)$ decreases by one and $R(x, v)$ increases by one. Finally, if the individual is of type \textbf{R} and $r < \alpha$, then $R(x, v)$ decreases by one and $S(x, v)$ increases by one. Again, the update is synchronous.

Then, we repeat the movement and transition steps until the end of the simulation. We finish the simulation before the boundary conditions become relevant. See Fig.~\ref{fig:evolution} for the evolution of the mesoscopic quantities, $\rho_j$.

\begin{figure}[t]
\centering\includegraphics[width=0.48\textwidth,height=0.2\textwidth]{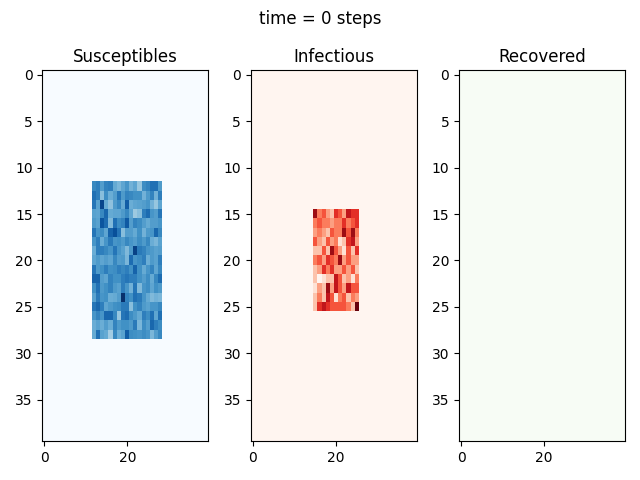}
\hfill
\includegraphics[width=0.48\textwidth,height=0.2\textwidth]{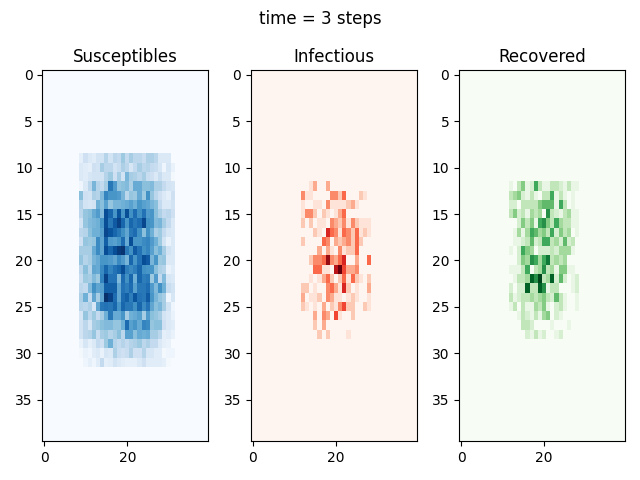}\\
\includegraphics[width=0.48\textwidth,height=0.2\textwidth]{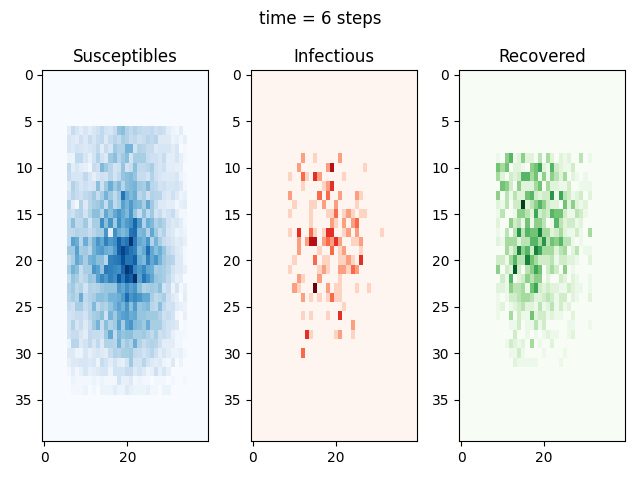}
\hfill    \includegraphics[width=0.48\textwidth,height=0.2\textwidth]{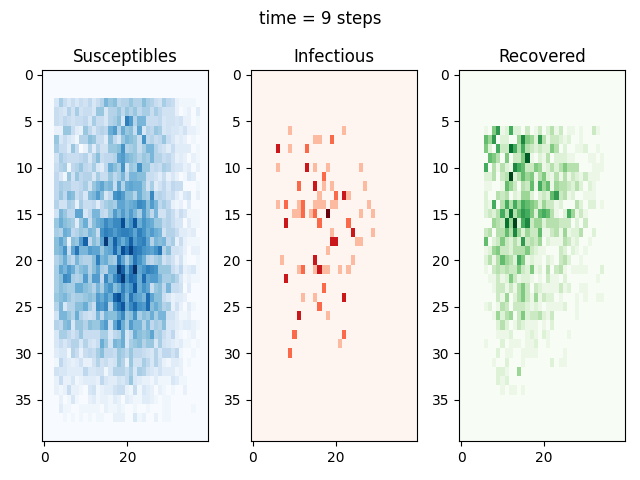}\\
\includegraphics[width=0.48\textwidth,height=0.2\textwidth]{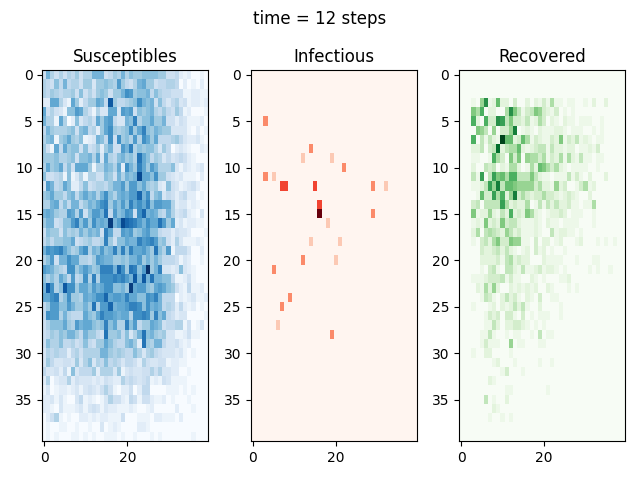}
\hfill    \includegraphics[width=0.48\textwidth,height=0.2\textwidth]{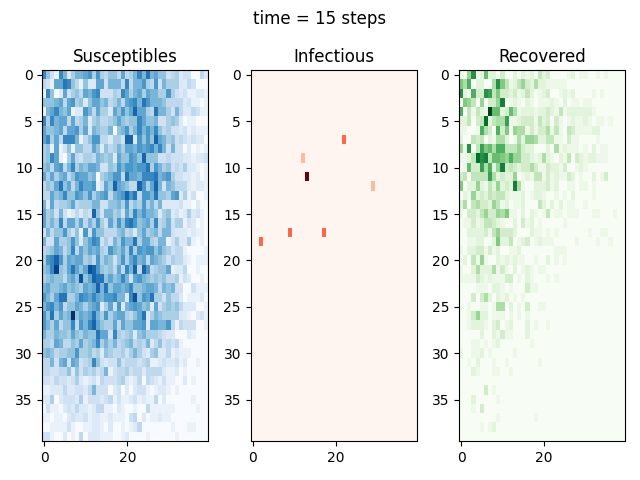}
\caption{Representation of the mesoscopic density $\rho_S$ (blue, left), $\rho_I$ (red, middle), and $\rho_R$ (right, green). More intense colors indicate higher concentrations. We consider the initial condition given by \textbf{S} and \textbf{I} individuals in the center of a square lattice of size $50\times 50$, with random initial velocities. Initial conditions are far from the boundary which are not reached for times considered. Initially, there are 10000 \textbf{S} individuals and 50 \textbf{I} individuals. We plot the configuration after 3, 6, 9, 12, and 15 steps, each step including the update of all particles. Class-transition parameters were given by $\alpha=0$ (i.e., immunity is permanent), $\beta=0.75$ (each encounter between \textbf{S} and \textbf{I} individuals has a probability of $\sfrac{3}{4}$ to generate a new \textbf{I} individual), and $\gamma=0.5$ (i.e., in each time step, each \textbf{I} individual has a probability $\sfrac{1}{2}$ to became \textbf{R}. The regression toward the mean parameter is $\lambda=0.5$, and we did not consider preferred directions, i.e., $\mathcal{Q}^{(P)}_j\equiv0$. }
\label{fig:evolution}
\end{figure}

\subsection{The constant parameters case}
\label{ssec:constant}

Assume that $\alpha,\beta,\gamma,\eta,\xi$ are positive constants. In this case, 
$\kappa_{j}=vF(|v|)/\lambda$, and $\Theta=(\eta+\xi)vF(|v|)/\lambda$, see Eqs.~\eqref{eq:kappa} and \eqref{eq:fS1}. The coefficients of the mesoscopic equations are given by
\[
D=\frac{1}{\lambda}\int_Vv^2F(|v|)\rd v\ ,\qquad
\Gamma=\frac{\eta+\xi}{\lambda}\int_Vv^2F(|v|)\rd v\ .
\]

The SIRS model will be, in this case, defined by parameters $\bar\alpha=\alpha$, $\bar\beta=\beta$, and $\bar\gamma=\gamma$, cf. Subsec.~\ref{ssec:SIRlimit}.

If we assume further that the space of admissible velocities is a hypersphere of radius $v_0$, for a movement in the two- or three-dimensional space, respectively, then $D=v_0^2/\lambda$ and $\Gamma=(\eta+\xi)v_0^2/\lambda$.

\subsection{Different disease risk acceptances}
\label{ssec:skeptical}

Let us assume that there is a proper subset $\Omega_0\subset\Omega$ such that $\eta=\xi=\eta_0$ for $x\in\Omega_0$ and $\eta=\xi=\eta_1$ otherwise, with constants $\eta_0<\eta_1$. This means that in the region $\Omega_0$, \textbf{S} individuals are more risk-prone to the disease, with contrary behavior outside $\Omega_0$. 

The interesting point is that the definition of $\kappa$ and $\Theta$, and, by extension, of $D$ and $\Gamma$ are local in $(t,x)$, and therefore the mesoscopic limits in each point $x\in\Omega$ will be obtained as if all $\eta$ and $\xi$ were constants in the entire domain with values equal to its values at $x$. 

\subsection{Confining infectious individuals}
\label{ssec:confining}

In this example, contrary to what was discussed so far, we consider that \textbf{S}, \textbf{I}, and \textbf{R} individuals have different velocity sets, allowing a possible confinement (voluntary or imposed) of \textbf{I} individuals. However, we will consider only the one-dimensional problem, as the dealing with class-transition kernels for different sets $V$ in the general multi-dimensional case is outside the scope of the present work. 

Given positive values $v_S,v_I,v_R>0$, we define three sets $V_S\bydef\{\pm v_S\}$, $V_I\bydef\{\pm v_I\}$, $V_R\bydef\{\pm v_R\}$, representing the set of possible velocities for \textbf{S}, \textbf{I}, and \textbf{R} individuals, respectively. Density of individuals will be given by $f_j^\pm=f_j(t,x,\pm v_j)$, with the superscript sign indicating if the corresponding individuals move to the right or to the left. Mesoscopic densities are given by $\rho_j=f_j^-+f_j^+$.

The equilibrium distribution is given by $F(v)=\frac{1}{2}$, and we proceed formally as in the Subsec.~\ref{ssec:mesoscopic}. In particular, the leading order is given by
\[
0=\mathcal{Q}^{(R)}_j[\mathbf{f}]=\lambda_j\left(\frac{f_{j,0}^++f_{j,0}^-}{2}-f_{j,0}^\pm\right)\ ,
\]
which implies that $f_{j,0}^+=f_{j,0}^-=\frac{\rho_j}{2}$. For the order 1 equation, we need to solve
\[
\lambda_j\left(\frac{\kappa^+_j-\kappa^-_j}{2}\right)=\mathcal{Q}^{(R)}_j[\kappa_j]=-\frac{v_j}{2}\ ,
\]
on the assumption that $\kappa_j^++\kappa_j^-=0$. We conclude that $\kappa_j^\pm=\pm\frac{v_j}{2\lambda_j}$, and, therefore $D_j=v_j^2/\lambda_j$. Finally, 
\[
\Theta_j=\frac{(\eta+\xi)v_j}{2\lambda_j}\delta_{jS}\ ,
\]
where $\delta_{SS}=1$ and $\delta_{jS}=0$ otherwise. From this, we conclude that 
\[
\Gamma_j=\frac{(\eta+\xi)v_j^{2}}{\lambda_j}\delta_{jS}\ .
\]

Mesoscopic equations are given by
\begin{align*}
&\partial_t\rho_S-\nabla_x\cdot\left(\frac{(\eta+\xi)v_S^2}{\lambda_S}(\nabla_x\rho_I)\rho_S\right)-\nabla_x\cdot\left(\frac{v_S^2}{\lambda_S}\nabla_x\rho_S\right)=\mathcal{Q}^{(T)}_S\left[\frac{\bm{\rho}}{2}\right]\ ,\\
    &\partial_t\rho_j-\nabla_x\left(\frac{v_j^2}{\lambda_j}\nabla_x\rho_j\right)=\mathcal{Q}^{(T)}_j\left[\frac{\bm{\rho}}{2}\right]\ ,\qquad j=I,R\ ,
\end{align*}
where $\bm{\rho}=(\rho_S,\rho_I,\rho_R)$.

We finish by noting that if $v_I\to0$, then the diffusion coefficient of \textbf{I} individuals converges to 0. Therefore, in this limit case, we expect that if $\rho_I(0,x)=0$, then $\rho_I(t,x)=0$ for all $t>0$. 

\section{Conclusions and future work}
\label{sec:conclusions}

We begin this work by introducing a kinetic model to describe the pathogen dynamics within a population. We considered a population of individuals that move freely in a bounded domain. Contagion occurs when susceptible and infectious individuals are in close proximity for an extended period of time. In other words, pathogen transmission only occurs if an \textbf{S} and an \textbf{I} individual are at the same point of the domain and moving with the same velocity. We started by showing that the kinetic model is formally compatible SIRS system of ordinary differential equations, i.e., the former simplifies to the latter when dealing with solutions depending only on time. 

The main advantage of using kinetic models is that they allow spatial phenomena, such as individual motility and spatial heterogeneity, to be included in the modelling. However, there are significantly more possibilities, and the present work has only begun to explore them.

One interesting application to be explored in future work is the introduction of vaccination, specifically allowing each individual in the population to freely choose whether to be vaccinated. The interaction between human behaviour with respect to vaccination and pathogen dynamics has been studied mathematically since the seminal work of~\cite{Bauch_Earn_2004}; see also~\cite{Doutoretal_JMB2016,Chalubetal_MB2024}. Unlike models based on ordinary differential equations, spatially inhomogeneous models allow us to consider local and global information simultaneously, both of which are relevant to each individual's decision on the central question: "To be or not to be vaccinated?" Furthermore, local and global information can produce conflicting answers to this question. An explicit example will help to illustrate this point. Consider a new class of individuals: the vaccinated ones. Assume that vaccination is voluntary. In a minimalistic model, human behaviour depends on global information, i.e. the endemic situation in the entire population, and local information given by the densities $f_j$, where $j=S,I,R,V$. Clearly, at least transiently, an outbreak in one region can coexist with a disease-free region. Kinetic models in epidemiology may provide a richer description of reality by considering the attenuation of information due to distance and the use of non-physical variables, e.g. describing individual risk acceptance.

A second approach is to explore more general transition kernels. In the current work, the transition kernel was specifically designed to correspond to the SIRS model at a macroscopic level. However, any system of ordinary differential equations in the form of the replicator system can be obtained as the macroscopic limit of a suitably defined class transition kernel. The replicator system of differential equations is a key model in evolutionary game theory (see, for example,~\cite{HofbauerSigmund}). In this work, we built upon a recent reformulation of the SIRS model in replicator equation form (cf.~\cite{HansenChalub_JTB2024}). To the best of our knowledge, the full relationship between kinetic models and game theory remains unexplored; however, see~\cite{Ambrosio_etal_2021}.

An important classical topic in the literature of kinetic models for chemotaxis is the existence of finite-time blow up \emph{versus} global-in-time solution; in particular, a finite-time blow up may occur in the partial differential equation regime (mesoscopic level, in our parlance), i.e., in the Keller-Segel system of partial differential equations~\cite{KellerSegel_70}. However, a formally compatible kinetic model may exhibit global existence~\cite{DolbeautPerthame,Chalubetal_04}. This question, which is extremely interesting for certain kernels inspired by game theory, is a current topic of interest for future research. Let us be more specific.

In the present work, we considered that \textbf{S} individuals turn into \textbf{I} by proximity, and that \textbf{S} tend to escape from higher concentrations of \textbf{I} individuals, or are neutral to higher concentrations of \textbf{I} individuals. We proved existence and uniqueness of solutions of the kinetic model in the case $\mathcal{Q}^{(P)}_j\equiv 0$, i.e. $\eta,\xi=0$, however, we expect that $\eta,\xi>0$ will have a regularising effect on the solution of the kinetic model, and Thm.~\ref{thm:existence} will still hold in this case. The situation would be completely different if we allowed $\eta$ and/or $\xi$ to be negative, i.e., if individuals of class \textbf{X} change to \textbf{Y} and are attracted by high concentrations of \textbf{Y} individuals.  Assume an evolutionary dynamics in which individuals with higher payoff reproduce at a higher rate and a rational behaviour in which they are attracted by environments in which they can expect an increase in payoff. These assumptions are natural in evolutionary game theory. In games of dominance or coordination games, one type will fixate while the other will become extinct~\cite{HofbauerSigmund}. In these cases, we expect a richer picture in which not only would finite-time blow up be possible, but also pattern formation, particularly in coordination games, as both types may fixate in different regions. The final outcome will possibly depend on the initial condition and deserves investigation. The general kinetic model with game-theoretical class transition kernels approach could be of interest in the study of pattern formation mediated by games played between cells, cf.~\cite{Tomlinson_EJC98,Dinglietal_BJC2009} for applications of evolutionary game theory in cancer modelling, or on the evolution of cooperation between rational agents, cf.~\cite{Axelrod_Hamilton:1981} for the initial studies on that topic and~\cite{NowakBonhoeffer_94} for its extension to spacial games.

From a purely mathematical point of view, one issue that requires further attention is the proof of convergence of $f_j^{(\varepsilon)}\stackrel{\varepsilon\to0}{\to}\rho_j$, which was not established in this study. This is a necessary step in proving the rigorous compatibility between the kinetic and mesoscopic models and should be considered in future work. 

For a correct modelling, the kernel $\mathfrak{Q}_j^{(\varepsilon)}$ required expansion of the dimensionless solution of the kinetic equation~\eqref{eq:kinetic_epsilon} in the parameter $\varepsilon$ up to order $\varepsilon^2$, unlike the case of the derivation of the Navier-Stokes formula from the Boltzmann equation with a collisional kernel~\cite{BardosGolseLevermore} or in the derivation of the Keller-Segel model from kinetic models for chemotaxis~\cite{OthmerHillen,Chalubetal_04}. The order 2 equation, studied in Subsubsec.~\ref{sssec:order2}, would be trivial if there were no class-transition in the model. The fact that $\mathcal{Q}^{(T)}_j$ is not conservative seems to be related to the necessity of the order $\varepsilon^2$ expansion. Furthermore, the effect of class transition is smaller than both the diffusion (order 0) and the drift (order 1), a fact that deserves further exploration; see also~\cite{Bisi_etal} for a recent example involving non-conservative kernels for which expansion up to order $\varepsilon^2$ is required to calculate formal macroscopic limits.

This paper has only begun to explore the link between kinetic models, mathematical epidemiology and, more broadly, evolutionary game theory. However, there is significant potential for further exploration of these interconnections!  

\section*{Acknowledgement}
All the authors are funded by national funds through the FCT – Funda\c{c}\~ao para a Ci\^encia e a Tecnologia, I.P., under the scope of the projects UIDB/ 00297/2020 (https://doi.org/10.54499/UIDB/00297/2020) and UIDP/00297/ 2020 (https:doi.org/10.54499/UIDP/00297/2020) (Center for Mathematics and Applications --- NOVA Math) and the project \emph{Mathematical Modelling of Multi-scale Control Systems: applications to human diseases} (CoSysM3)  2022.03091.PTDC (https://doi.org/10.54499/2022.03091.PTDC), supported by national funds (OE), through FCT/MCTES. All the authors thank Ana Jacinta Soares (Universidade do Minho, Portugal) for insightful comments on preliminary versions of the present work. FACCC also acknowledges Robert Pego (Carnegie Mellon University) for suggesting the application of the current theory in the 1D case, with varying velocity sets, resulting in the example discussed in Subsec.~\ref{ssec:confining}.

All the authors contributed equally to the development of the work's ideas, computational codes, data analysis, discussions, and writing of the final version of the manuscript.


\newcommand{\etalchar}[1]{$^{#1}$}

 \end{document}